\documentclass[11pt,reqno]{amsart}

\usepackage{imports}
\usepackage{subcaption}

\title{WalkSAT is linear on random 2-SAT}

\thanks{Amin Coja-Oghlan is supported by DFG CO 646/3, DFG CO 646/5 and DFG CO 646/6.
Pavel Zakharov is supported by DFG CO 646/6.
Petra Berenbrink, Colin Cooper, and Thorsten G\"otte are supported by DFG grant no.\ 491453517.
Petra Berenbrink and Lukas Hintze are supported by DFG grant no.\ 411362735.}

\author{Petra Berenbrink, Amin Coja-Oghlan, Colin Cooper, Thorsten G\"otte, Lukas Hintze, Pavel Zakharov}

\address{Petra Berenbrink, {\tt berenbrink@informatik.uni-hamburg.de}, University of Hamburg, Faculty of Mathematics, Informatics and Natural Sciences, Department of Informatics, Vogt-K\"olln-Str.\ 30, 22527 Hamburg, Germany.}
\address{Amin Coja-Oghlan, {\tt amin.coja-oghlan@tu-dortmund.de}, TU Dortmund, Faculty of Computer Science and Faculty of Mathematics, 12 Otto-Hahn-St, Dortmund 44227, Germany.}
\address{Colin Cooper, {\tt colin.cooper@kcl.ac.uk}, Department of Informatics, King’s College, University of London, London WC2B 4BG, UK.}
\address{Lukas Hintze, {\tt lukas.rasmus.hintze@uni-hamburg.de}, University of Hamburg, Faculty of Mathematics, Informatics and Natural Sciences, Department of Informatics, Vogt-K\"olln-Str.\ 30, 22527 Hamburg, Germany.}
\address{Thorsten G\"otte, {\tt thorsten.goette@uni-hamburg.de}, University of Hamburg, Faculty of Mathematics, Informatics and Natural Sciences, Department of Informatics, Vogt-K\"olln-Str.\ 30, 22527 Hamburg, Germany.}
\address{Pavel Zakharov, {\tt pavel.zakharov@tu-dortmund.de}, TU Dortmund, Faculty of Computer Science and Faculty of Mathematics, 12 Otto-Hahn-St, Dortmund 44227, Germany.}

\newcommand\WS{\texttt{WalkSAT}}

\begin{document}

\begin{abstract}
	In an influential article Papadimitriou [FOCS~1991] proved that a local search algorithm called \WS\ finds a satisfying assignment of a satisfiable 2-CNF with $n$ variables in $O(n^2)$ expected time.
	Variants of the \WS\ algorithm have become a mainstay of practical SAT solving (e.g., [Hoos and St\"utzle~2000]).
	In the present article we analyse the expected running time of \WS\ on random 2-SAT instances.
	Answering a question raised by Alekhnovich and Ben-Sasson [SICOMP~2007], we show that \WS\ runs in linear expected time for all clause/variable densities up to the random 2-SAT satisfiability threshold.
\end{abstract}

\maketitle

\section{Introduction}\label{S:intro}

\subsection{Background and motivation}
Since the 1990s random instances of random constraint satisfaction problems (`CSPs') such as $k$-SAT have been considered challenging benchmark instances~\cite{Cheeseman}.
In particular, experiments indicate that the running time of DPLL-like algorithms increases near the {\em satisfiability threshold}, i.e., the critical ratio $m/n$ of clauses to variables up to which a random formula likely remains satisfiable~\cite{MSL}.
These observations led to two research questions: to pinpoint the satisfiability thresholds of random CSPs, and to analyse the performance of algorithms on random problems~\cite{ANP}.

Both of these tasks turned out to be remarkably challenging. 
While the picture remains incomplete, much progress has been made on the satisfiability threshold problem.
For instance, a line of work~\cite{AM,AP,COP} culminated in the article by Ding, Sly and Sun~\cite{DSS} that verified a statistical physics-inspired conjecture as to the precise location of the random $k$-SAT satisfiability threshold for large $k$.

Perhaps surprisingly, with respect to the second task of analysing satisfiability algorithms, even the random $2$-SAT problem still poses challenges.
The analysis of the `simple' \WS\ algorithm on random $2$-SAT is a case in point.
We recall that \WS\ operates as follows~\cite{WalkSAT1,WalkSAT2}.
\begin{quote}
	Initially set all variables to `true'.
	While there exist unsatisfied clauses, pick one of these clauses randomly, then pick a random variable in that clause, flip the truth value assigned to the variable and continue.%
	\footnote{Instead of starting from the all-true assignment one could as well start from a random assignment. In any case, on random inputs both version of the algorithm are equivalent.}
\end{quote}
It is well known that on satisfiable 2-SAT formulas with $n$ variables this algorithm returns a satisfying assignment within an expected $O(n^2)$ running time~\cite{WalkSAT1}; the pretty proof is frequently covered in classes on randomised algorithms.
However, it has been conjectured~\cite{AB}, even `predicted' based on physics methods~\cite{SM}, that \WS\ actually succeeds in {\em linear} expected time on random 2-SAT formulas right up to the satisfiability threshold (see \Sec~\ref{sec_discussion} for a detailed discussion).
The purpose of the present paper is to prove this conjecture.

\subsection{The main result}
To state the main result precisely, let $\PHI=\PHI_{n, m}$ be a random 2-CNF on $n$ variables $x_1,\ldots,x_n$ with $m$ clauses, where each clause is drawn independently and uniformly from the set of all $4n (n-1)$ possible 2-clauses.
(Thus, independently for each clause we pick a random ordered pair of distinct variables and place negations independently with probability $1/2$.)
Throughout the paper we let $m\sim\alpha n$ for a real $\alpha>0$.
Furthermore, we say that the random formula $\PHI$ enjoys a property $\cE$ {\em with high probability} (`\whp') if $\lim_{n\to\infty}\pr\brk{\PHI\in\cE}=1$.

It has been known since the 1990s that $\alpha=1$ is the random 2-SAT satisfiability threshold.
Thus, $\PHI$ is satisfiable \whp if $\alpha<1$, and unsatisfiable \whp if $\alpha>1$~\cite{CR,Goerdt}.
Indeed, throughout the satisfiable regime $\alpha<1$ the random formula $\PHI$ possesses an exponential number of satisfying assignments~\cite{2sat}.

Let $\vT(\Phi)\in\{0,1,2,\ldots\}\cup\{\infty\}$ be the (random) number of times that \WS\ flips a variable on a given input formula $\Phi$.
Depending on the random choices of \WS, $\vT(\Phi)$ is a random variable even if $\Phi$ is non-random.
Hence, to separate the randomness of the algorithm from the random choice of the formula $\PHI$, we study the expected running time $\ex[\vT(\PHI)\mid\PHI]$ given $\PHI$, where the expected value is taken over the random choises made by \WS\ algorithm only.
The following theorem shows that \whp over the choice of $\PHI$, this conditional expectation is linear in $n$ throughout the satisfiable regime.

\begin{theorem}\label{thm_main}
	There exists a constant $C>0$ such that for all $0\leq\alpha<1$ we have
	\begin{align}\label{eq_thm_main}
		\ex[\vT(\PHI)\mid\PHI]&\leq\frac{Cn}{(1-\alpha)^2}&&\mbox\whp
	\end{align}
\end{theorem}

Thus, \Thm~\ref{thm_main} shows that for any $\alpha$ below the satisfiability threshold, the expected running time of \WS\ scales linearly in the number $n$ of variables, with a prefactor that depends on $\alpha$ but remains fixed as $n\to\infty$.
By contrast, the worst-case bound from~\cite{WalkSAT1} yields an upper bound that is quadratic in $n$.
Furthermore, the best previous (albeit unpublished) analysis of \WS\ gave a superlinear bound of $O(n\log^6n)$ for $\alpha<0.99$~\cite{Dolle}.
Finally, for $\alpha<1/2$ (but not beyond) an $O(n)$ bound could be obtained easily because in this regime the formula $\PHI$ decomposes into mostly bounded-sized sub-formulas on disjoint variable sets; of course, this argument breaks down for $1/2<\alpha<1$.

After introducing a bit of notation we proceed to outline the proof of \Thm~\ref{thm_main} and subsequently discuss further related work and open questions.
The remaining sections then contain the proof details.

\subsection{Preliminaries and notation}\label{sec_pre}
Throughout the paper we encode the Boolean value `true' as $1$ and `false' as $-1$, thus an assignment of a variable set $V$ is a vector in $\{-1, 1\}^V$.
Additionally, it will turn out to be convenient to write $(+1)\cdot x$ for the literal $x$ and $(-1)\cdot x$ for the literal $\neg x$.
As usual, for a Boolean variable $x$ we denote by $\neg x$ the negation of $x$.
Moreover, a literal $l$ is either a Boolean variable $x$ or its negation $\neg x$.
The same notation extends to literals $l$; thus, $(+1)\cdot l\equiv l$ and $(-1)\cdot l\equiv\neg l$.
Conversely, for a literal $l\in\{x,\neg x\}$ we denote by $\abs l=x$ the underlying Boolean variable.
A 2-SAT formula or a 2-CNF formula $\Phi$ consists of a finite set $V(\Phi)$ of variables and a finite set $F(\Phi)$ of clauses, where each clause contains two distinct variables.

To fix further notation, we refer to the following \WS\ pseudocode.

\IncMargin{1em}
\begin{algorithm}[h!]\label{alg_ws}
    \KwData{A 2-CNF $\Phi$ on $n$ variables.}
	let $\SIGMA^{(0)} = \vecone$ be the all-true assignment and set $t=0$\;
	\While{$\SIGMA^{(t)}$ does not satisfy $\Phi$ and $t<100n^2$}{
		randomly and uniformly pick a clause $\va^{(t)} = \vl_1^{(t)} \vee \vl_2^{(t)}$ that $\SIGMA^{(t)}$ fails to satisfy \;
		choose $\vh^{(t)} \in \{1, 2\}$ uniformly\;
		obtain $\SIGMA^{(t+1)}$ from $\SIGMA^{(t)}$ by flipping the value of variable $|\vl_{\vh^{(t)}}^{(t)}|$\;
		increase $t$ by one\;
    }
	\Return\ $\SIGMA^{(t)}$.
    \caption{\WalkSAT}\label{fig_walksat}
\end{algorithm}
\DecMargin{1em}

We tacitly assume $n$ is sufficiently large for our various estimates to hold.
Asymptotic notation such as $O(\nix),o(\nix)$ etc.\ always refers to the limit $n\to\infty$.

\section{Proof strategy}

\noindent
The classical method for analysing linear time algorithms on random formulas is the so-called `method of deferred decisions', beautifully explained, e.g., in~\cite{Achlioptas01}.
The basic idea is to trace the progress of the algorithm step by step in such a way that the parts of the formula pertinent to the execution of the current step remain random subject to some suitable conditioning.
For example, one might condition on the number of clauses with a specific number of as yet unassigned variables.
Then the relevant random variables that determine the progress of the algorithm can typically be tracked via a system of ordinary differential equations~\cite{Wormald}, thus turning the combinatorial problem into an exercise in calculus.
Examples of this analysis technique (with certain refinements) include the work of \Chvatal\ and Reed~\cite{CR} and of Frieze and Suen~\cite{FriezeSuen}.

A necessary condition for the method of deferred decisions to apply is that the algorithm proceeds essentially linearly, i.e., assigns one variable at a time and never, or at least very rarely, revises prior decisions.
Unfortunately, \WS\ grossly violates this necessary condition.
Due to its random choices, the algorithm constantly revisits variables and is bound to flip their values many times over.
Therefore, the analysis of \WS\ requires new ideas.

Building upon insights from~\cite{2sat}, we will combine a probabilistic and a deterministic argument to prove \Thm~\ref{thm_main}.
Specifically, we will show that \whp around each variable two sub-formulas of the random formula $\PHI$ can be identified that `lock \WS\ in'.
This means that each of the two sub-formulas has the property that once \WS\ finds a satisfying assignment of the sub-formula, \WS\ is not going to flip any of the sub-formula's variables anymore.
We refer to these sub-formulas as {\em implication sub-formulas}.
The probabilistic part of the analysis consists in establishing a reasonably small bound on the total sizes of the implication sub-formulas.
The deterministic part consists in bounding the running time of \WS\ in terms of this total size.

A key tool to construct the implication sub-formulas is the Unit Clause Propagation algorithm, with which we start.
Subsequently we will verify that the implication sub-formulas do indeed `lock \WS\ in'.
Finally we derive the desired running time bound.

\subsection{Unit Clause Propagation and implication sub-formulas}

Unit Clause Propagation (`\UCP') can be viewed as an algorithm that iteratively tracks logical implications. 
More precisely, the following version of the procedure is given a set of literals $\cL$, which are deemed to be true.
The \UCP\ algorithm then enhances the set $\cL$ by pursuing direct logical implications necessary to avoid violated clauses.

\IncMargin{1em}
\begin{algorithm}[h!]
    \KwData{A 2-CNF $\Phi$ along with a set $\cL$ of literals deemed true.}
    \While{\upshape there exists a clause $a\equiv \neg l\vee l'$ with $l\in\cL$ and $l'\not\in\cL$}{%
        add literal $l'$ to $\cL$\;
        add $a$ to $\cC$\;}
    \Return\ $\cL,\cC$\;
    \caption{Unit Clause Propagation (`\UCP').}\label{fig_ucp}
\end{algorithm}
\DecMargin{1em}

We will use \UCP\ to construct the aforementioned implication sub-formulas.
Thus, let $\Phi$ be a 2-CNF formula. 
For a literal $l$ let $\cL(\Phi, \{l\})$ be the set of literals produced by \UCP$(\Phi, \{l\})$ and let $\cV(\Phi, \{l\})$ be the corresponding set of variables.
Also, we write $\cC(\Phi, \{l\})$ for the set of clauses produced by \UCP$(\Phi, \{l\})$.
Further, let $\Phi^{(l)}$ be the sub-formula comprising of $\cV(\Phi,\cbc l)$ and $\cC(\Phi,\{l\})$.
Finally, for a given formula $\Phi$ let
\begin{align}\label{eql2}
	X(\Phi)&=\sum_{x\in V(\Phi)}|\cV(\Phi,\cbc{x})|^2+|\cV(\Phi,\cbc{\neg x})|^2
\end{align}
be the sum of squares of the sizes of the implication sub-formulas.
The following statement bounds $X(\PHI)$ for the random formula $\PHI$.
The proof of \Prop~\ref{prop_l2}, which builds upon arguments from~\cite{2sat}, can be found in \Sec~\ref{sec_prop_l2}.

\begin{proposition}\label{prop_l2}
	There exists a constant $C>0$ such that for $0<\alpha<1$ \whp we have
	\begin{align*}
		X(\PHI)&\leq\frac{Cn}{(1-\alpha)^2}.
	\end{align*}
\end{proposition}

\subsection{Locking \WS\ in}
We continue with a deterministic statement to the effect that once \WS\ finds a satisfying assignment of an implication sub-formula $\Phi^{(l)}$ with $l$ set to true, the algorithm will not subsequently flip any variable from $\cV(\Phi,\cbc l)$.
Recall that the random variable $\vT(\nix)\geq0$ equals the total number of flips that \WS\ performs.
Also recall from \Sec~\ref{sec_pre} that $\SIGMA^{(t)}$ is the assignment that \WS\ obtained after the first $t$ flips.
Finally, for a literal $l$ and an assignment $\sigma$ let $\fE(\Phi, l, \sigma)$ be the predicate that all literals from $\cL(\Phi, \{l\})$ are true under $\sigma$.

We observe that if we take some satisfying assignment $\sigma$ of a formula $\Phi$, then the implication subformulas of all literals $l$ that $\sigma$ sets to true are satisfied.

\begin{fact}\label{claim_sat_2}
    Let $\Phi$ be a satisfiable 2-CNF formula and $\sigma^*$ be a satisfying assignment of $\Phi$.
	Then for all variables $x \in V(\Phi)$ the predicate $\fE(\Phi, \sigma^*(x)\cdot x , \sigma^*)$ holds.
\end{fact}
\begin{proof}
    Recall that $\sigma^*(x) \cdot x$ is the literal of $x$ which is deemed true under $\sigma^*$.
	Then, the assertion is immediate from the construction of $\Phi^{(l)}$ by way of the \UCP\ algorithm.
\end{proof}

Furthermore, we show that once $\fE(\Phi,l,\SIGMA^{(t)})$ occurs for a $0\leq t\leq\vT(\Phi)$, it will continue to occur at all subsequent times.

\begin{proposition}\label{prop_flips_per_var}
	Let $\Phi$ be a 2-CNF and let $l$ be a literal.
	If there exists $0 \leq t_0 \leq \vT(\Phi)$ such that the event $\fE(\Phi,l,\SIGMA^{(t_0)})$ occurs, then the event $\fE(\Phi,l,\SIGMA^{(t)})$ occurs for all $t_0<t\leq\vT(\Phi)$.
\end{proposition}

The proof of \Prop~\ref{prop_flips_per_var} can be found in \Sec~\ref{sec_prop_flips_per_var}.
Combining Fact~\ref{claim_sat_2} and \Prop~\ref{prop_flips_per_var}, we see that once \WalkSAT\ satisfies the implication subformula $\Phi^{(l)}$ of some literal $l$, it will not flip any of its variables at any later time.

\subsection{The expected running time}

For a 2-CNF $\Phi$ and a literal $l$ let
\begin{align*}
	\vN(\Phi, l) = \sum_{v \in \cV(\Phi, \{l\})}\sum_{t=1}^{\vT(\Phi)} \vecone\cbc{\SIGMA^{(t-1)}(v)\neq\SIGMA^{(t)}(v)}
\end{align*}
be the total number of times  that \WS$(\Phi)$ flips variables from $\cV(\Phi, \{l\})$.
We make a note of the following running time bound.

\begin{fact}\label{claim_sat_3}
    Let $\Phi$ be a satisfiable 2-CNF and $\sigma^*$ be a satisfying assignment.
	Then
	\begin{align}\label{eq_claim_sat_3}
			\vT(\Phi)\leq\sum_{x \in V} \vN(\Phi,\sigma^*(x) \cdot x  ).
		\end{align}
\end{fact}
\begin{proof}
	The proof utilizes the fact that $x \in  \cV(\Phi, \sigma^*(x) \cdot x)$ for any $x \in V$. Clearly,
	\begin{align*}
		\vN(\Phi, \sigma^*(x) \cdot x) \geq \sum_{t=1}^{\vT(\Phi)} \vecone\cbc{\SIGMA^{(t-1)}(x)\neq\SIGMA^{(t)}(x)},
	\end{align*}
	for any $x \in V$ and
	\begin{align*}
		\sum_{x \in V} \vecone\cbc{\SIGMA^{(t-1)}(x)\neq\SIGMA^{(t)}(x)} = 1
	\end{align*}
	for any $t = 1, \ldots, \vT(\Phi)$. Finally,
	\begin{align*}
		\sum_{x \in V} \vN(\Phi,\sigma^*(x) \cdot x  ) \geq \sum_{x \in V} \sum_{t=1}^{\vT(\Phi)} \vecone\cbc{\SIGMA^{(t-1)}(x)\neq\SIGMA^{(t)}(x)} = \sum_{t=1}^{\vT(\Phi)} \sum_{x \in V} \vecone\cbc{\SIGMA^{(t-1)}(x)\neq\SIGMA^{(t)}(x)} = \sum_{t=1}^{\vT(\Phi)} 1 = \vT(\Phi),
	\end{align*}
	which completes the proof.
\end{proof}

As a next step we estimate the r.h.s.\ of~\eqref{eq_claim_sat_3} in terms of the sizes of the implication sub-formulas.
The proof of the following proposition can be found in \Sec~\ref{sec_prop_upper_bound_small_subformula}.

\begin{proposition}\label{prop_upper_bound_small_subformula}
	There is a constant $C>0$ such that the following is true.
    Let $\Phi$ be a satisfiable 2-CNF formula and $\sigma^*$ be a satisfying assignment.
	Then
	\begin{align*}
		\Erw \brk{\vN(\Phi,\sigma^*(x)\cdot x)} &\leq C\cdot\abs{\cL(\Phi, \{\sigma^*(x)\cdot x)\})}^2&&\mbox{for all }x\in V(\Phi).
	\end{align*}
\end{proposition}

The crucial observation towards proving \Prop~\ref{prop_upper_bound_small_subformula} is that everything that happens outside of the implication subformula of some literal can only improve the situation (i.e.\ make some of the literals satisfied).
The last step is to estimate the above quantity for the random 2-SAT formula, which completes the proof of \Thm~\ref{thm_main}.

\begin{proof}[Proof of \Thm~\ref{thm_main}]
	Given $0<\alpha<1$ let $\fA$ be the event that $\Phi$ is satisfiable and $X(\Phi)\leq C(1-\alpha)^{-2}n$.
	Since $\alpha$ is below the satisfiability threshold, \Prop~\ref{prop_l2} implies that
	\begin{align}\label{eq_thm_main1}
		\pr\brk{\PHI \in \fA}&=1-o(1).
	\end{align}
	Suppose $\Phi\in\fA$ and let $\SIGMA^* = \SIGMA^*_{\PHI}$ be a satisfying assignment of $\PHI$.
    Then Fact~\ref{claim_sat_3}, \eqref{eql2} and \Prop~\ref{prop_upper_bound_small_subformula} imply that
	\begin{align}\label{eq_thm_main2}
        \ex\brk{T(\PHI)\mid\PHI = \Phi}\leq \Erw\brk{\sum_{x \in V} \vN(\PHI,  \SIGMA^*(x) \cdot x  )\mid\PHI = \Phi} \leq X(\Phi)\leq\frac{Cn}{(1-\alpha)^2}.
	\end{align}
	Finally, the assertion follows from \eqref{eq_thm_main1}--\eqref{eq_thm_main2}:
	\begin{align*}
		\Pr\brk{ \Erw\brk{\vT(\PHI) \mid \PHI} \leq {Cn \over \bc{1 - \alpha}^2}} \geq \Pr\brk{ \Erw\brk{\vT(\PHI) \mid \PHI} \leq {C n\over \bc{1 - \alpha}^2} \mid \PHI \in \fA} \cdot \Pr\brk{\PHI \in \fA} = 1 - o(1).
	\end{align*}
\end{proof}

\subsection{Discussion}\label{sec_discussion}
The random 2-SAT satisfiability threshold was determined independently and simultaneously by Goerdt~\cite{Goerdt} and \Chvatal\ and Reed~\cite{CR}.
While Goerdt's proof is based on counting certain implication chains called `bicycles', the lower bound proof of \Chvatal\ and Reed is algorithmic.
Specifically, \Chvatal\ and Reed prove that a variant of the Unit Clause Propagation algorithm solves random 2-SAT instances in linear time \whp for all clause/variable densities up to the satisfiability threshold.
Another linear time algorithm for constructing satisfying assignments of random 2-CNFs up to the satisfiability threshold can be based on the so-called {\em pure literal heuristic}.
The idea behind it is simple: if a formula contains literal $l$ but not literal $\neg l$, then no harm can be done by setting $l$ to `true'.
This assignment satisfies all clauses in which $l$ appears, which can thus be removed from the formula, which may create additional pure literals.
This process will \whp satisfy all but a bounded number of clauses of a random 2-CNF.
The remaining clauses form short cycles \whp that can be satisfied individually~\cite{2sat,CR}.

Irrespective of the existence of other linear time algorithms, the rigorous study of local search algorithms in general and \WS\ in particular remains an outstanding challenge.
Indeed, despite its simplicity, the \WS\ algorithm has turned out to be remarkably difficult to analyse on random $k$-SAT instances for $k\geq2$.
The important article of Alekhnovich and Ben-Sasson~\cite{AB} contributed a tool called `terminators' in order to estimate the running time of \WS.
Essentially terminators are weighted satisfying assignments.
A bound on the running time of \WS\ can be obtained by multiplying their $\ell_1$-norm and their $\ell_\infty$-norm~\cite[\Thm~3.2]{AB}.
Alekhnovich and Ben-Sasson further define the {\em terminator threshold} as the largest clause/variable density up to which terminators exists such that this product is of order $O(n)$; up to this threshold \WS\ thus runs in linear expected time.
For $k=3$ Alekhnovich and Ben-Sasson proved that the terminator threshold is lower bounded by the pure literal threshold $\alpha_{3,\mathrm{pure}}$ up to which the pure literal rule produces a satisfying assignment.
Broder, Frieze and Upfal~\cite{BFU} had previously established that $\alpha_{3,\mathrm{pure}}\approx 1.63$.

Regarding random 2-SAT, Alekhnovich and Ben-Sasson conjectured that the ``terminator threshold'', i.e., the clause/variable density up to which terminators exist that demonstrate a linear running time of \WS\ on random 2-CNFs, matches the satisfiability threshold $1=\alpha\sim m/n$.
They suggested the proof of this conjecture as a strategy to proving that \WS\ solves random 2-CNFs in linear expected time up to the satisfiability threshold.
This conjecture/proof strategy may seem appealing because in random 2-SAT the satisfiability threshold coincides with the pure literal threshold~\cite{2sat}.
However, the techniques from~\cite{AB} do not suffice to prove the existence of terminators near the satisfiability threshold $\alpha=1$ as a key expansion argument only works for $k\geq3$~\cite[Appendix~A]{AB}.
Hence, in the present paper we establish the linear expected running time of \WS\ directly, without the detour via terminators.
Instead parts of the proof, particularly the proof of \Prop~\ref{prop_l2}, build upon branching process arguments first developed in  recent work on the number of satisfying assignments of random 2-CNFs~\cite{2sat,2satclt}.

\begin{figure}\centering
    \subcaptionbox{Normalized runtime for $\alpha \in \{0.1, 0.3, 0.5, 0.7, 0.9\}$ and 512 values of $n \in [2^{10}, 2^{23})$ spaced geometrically.\label{fig:alpha_const}}{
        \includegraphics[width=0.48\linewidth,trim=0.3cm 0.6cm 1.4cm 1.6cm]{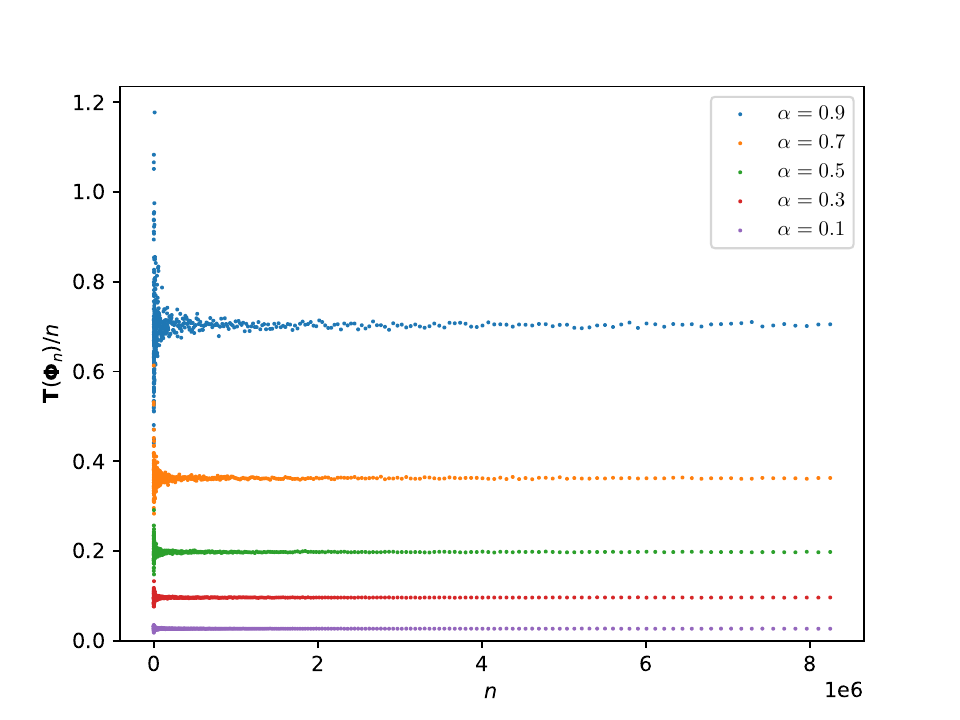}
    }
    \subcaptionbox{Normalized runtime for $n=10^8$ and 512 values of $\alpha \in [0.5, 1-2^{-10}]$ spaced evenly.\label{fig:n_const}}{
        \includegraphics[width=0.48\linewidth,trim=0.3cm 0.6cm 1.4cm 1.6cm]{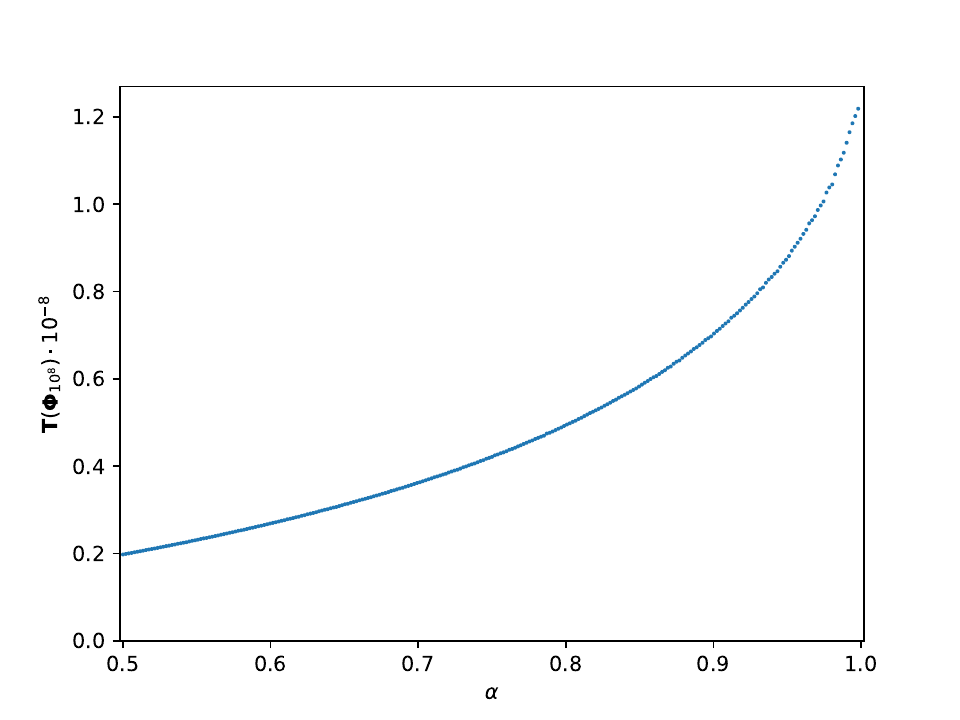}
    }
    \caption{Empirical runtime of \WS, normalized by $n$. Each point represents one run of \WS\ on a random $2$-SAT instance with $m=\alpha n$ clauses.}
    \label{fig:WS_runtime_plot}
\end{figure}

A heuristic, statistical physics based article of Monasson and Semerjian~\cite{SM} investigates the performance of \WS\ on random $k$-SAT instances with $k\geq2$.
Among other things they predict thresholds up to which \WS\ runs in linear expected time.
In the case $k=2$ their prediction matches the value $m/n\sim\alpha=1$.
\Thm~\ref{thm_main} vindicates this prediction.
An exciting open question remains to find the precise dependency of the running time on $\alpha$, which Semerjian and Monasson also consider.
Empirically, the runtime normalized by $n$ does seem to be tightly concentrated around a constant depending on $\alpha$ for large $n$ (cf.\ Figure \ref{fig:alpha_const}), with this constant increasing (albeit slowly) as $\alpha$ approaches $1$ (see Figure \ref{fig:n_const}).

Improving a prior bound from~\cite{COFFKV}, Coja-Oghlan and Frieze~\cite{COF} showed that for large values of $k$ \WS\ runs in linear expected time \whp for clause/variable densities up to $c\cdot 2^k/k$~\cite{COF} for an absolute constant $c>0$.
Up to the value of $c$ this result matches the prediction from~\cite{SM}.
Conversely, \WS\ is known to exhibit an exponential running time \whp for clause/variable densities beyond $c'\cdot 2^k\log^2k/k$ for another constant $c'>0$~\cite{COHH}.
By comparison, the best known polynomial time algorithm for random $k$-SAT succeeds up to densities $(1-\eps_k)2^k\log k/k$, where $\eps_k\to 0$ in the limit of large $k$~\cite{better}.
For comparison, the random $k$-SAT satisfiability threshold has a value of $2^k\log 2+O(1)$~\cite{COP,DSS}.

Local search algorithms inspired by the \WS\ idea play a prominent role in both worst-case exponential time algorithms for the $k$-SAT problem and in practical SAT solving.
A well known example of the former line of work is Sch\"oning's analysis of \WS~\cite{WalkSAT2}, which yields an expected running time of $(2(k-1)/k)^{n+o(n)}$ on satisfiable $k$-CNFs.
Moreover, significant contributions to the practical aspect of solving $k$-SAT problem were made by Hoos \cite{Hoos1, Hoos2, Hoos3}, who suggested and explored various local search algorithms based on \WS.

\section{Proof of \Prop~\ref{prop_l2}}\label{sec_prop_l2}

\noindent
We start with providing an upper bound on $\Erw[X(\PHI)]$.

\begin{lemma}\label{lem_exX}
	There is a constant $C>0$ such that $\ex[X(\PHI)]\leq C(1-\alpha)^{-2}n$.
\end{lemma}

To prove \Lem~\ref{lem_exX} we consider two regimes of $\alpha$ separately.
For the case $\alpha \geq 0.1$ we will use the following lemma.

\begin{lemma}[{\cite[\Lem~4.3]{2satclt}}]\label{lem_forced_tail}
    For any literal $l$ and every $t>4/(1-\alpha)$ we have
    \begin{align*}
        \pr\brk{\abs{\cV(\PHI,\{l\})}>t}\leq(2+o(1))\exp(-\alpha t/20).
    \end{align*}
\end{lemma}

For $\alpha < 0.1$\footnote{We can extend this argument up to $\alpha = 0.5$ (which corresponds to the emergence of the giant component in the underlying random hypergraph). However, for the purposes of the proof of \Lem~\ref{lem_exX}, it is sufficient to establish it for any positive $\alpha$.}
we will resort to the well-known branching process argument from the theory of random graphs; see~\cite{AS} for an excellent exposition.
We apply the argument to the graph $\Gamma(\Phi)$ constructed as follows, where $\Phi$ is a 2-CNF formula:
The vertices of $\Gamma(\Phi)$ are the variables of $\Phi$ and there is an edge between vertices $v_1$ and $v_2$ if there is a clause in $\Phi$ which contains $v_1$ and $v_2$ (regardless of signs).
(Readers familiar with the implication graph associated with a 2-CNF can see $\Gamma(\Phi)$ as obtained from the implication graph of $\Phi$ by disregarding the orientation of the edges.)
By $C(G, v)$ denote the connected component in $G$ that contains vertex $v$.
We make a note of the fact that
\begin{align}\label{eqAS1}
	|\cV(\Phi, \{l\})|&\leq|C(\Gamma(\Phi), |l|)|,
\end{align}
since being in the same implication sub-formula entails existence of a weak path between the two variables.

How is the graph $\Gamma(\PHI_n)$ associated with a random 2-CNF distributed?
Basically $\Gamma(\PHI_n)$ is a uniformly random graph with potential edge repetitions, although each edge can be repeated at most four times (for the four possible sign combinations).
Thus, the set $E(\Gamma(\PHI_n))$ of edges of $\Gamma(\PHI_n)$ is uniformly random given its size $|E(\Gamma(\PHI_n))|$.
In other words, the graph $\Gamma(\PHI_n)$ is nothing but a uniform \Erdos-\Renyi\ random graph $\GG(n,|E(\Gamma(\PHI_n))|)$ on $n$ vertices.
Of course, estimates of the component sizes of the \Erdos-\Renyi\ random graph are readily available.
Specifically, we will use the following theorem.%
\footnote{In {\cite[\Thm~11.6.1]{AS}} the result is stated in terms of the binomial random graph where edges appear independently with probability $p=c/n$. However, it is well known that the binomial random graph and the \Erdos-\Renyi\ model are interchangeable with respect to their component sizes~\cite{JLR}.}

\begin{theorem}[{\cite[\Thm~11.6.1]{AS}}]\label{thm_branching_est}
    Let $\GG(n,M)$ be the \Erdos-\Renyi\ random graph with $n$ vertices and $M\sim cn/2$ edges with $c<1$.
	For a vertex $v$ let $C(v)$ denote the connected component that contains $v$.
	Then for any $k > 0$
    \begin{align*}
        \Pr\brk{|C(v)| = k} = 
			(1 + o(1)){\eul^{-ck} \cdot (ck)^{k-1} \over k!} .
    \end{align*}
\end{theorem}

\begin{proof}[Proof of \Lem~\ref{lem_exX}]
In the case $\alpha\geq0.1$ we use \Lem~\ref{lem_forced_tail}, which implies that for any literal $l$,
\begin{align}\label{eq_lem_exX1}
	\Erw \abs{\cV(\PHI_n, \{l\})}^2 &\leq \bc{4 \over 1 - \alpha}^2 + \sum_{t=0}^{\infty} t^2 \cdot (2 + o(1)) \exp\bc{-\alpha t/20} \leq {16 \over (1-\alpha)^2} + {5 \over (\eul^{1/200} - 1)^3} \leq {10^{8} \over (1-\alpha)^2},
\end{align}
and thus
\begin{align}\label{eq_lem_exX1a}
	\Erw \brk{X(\PHI)} &\leq  {{2\cdot 10^{8}n} \over (1-\alpha)^2}.
\end{align}
Moreover, for $\alpha<0.1$ we combine~\eqref{eqAS1} with \Thm~\ref{thm_branching_est}.
Specifically, let $c=0.3$ and thus $M\sim 0.15n$ so that $|E(\Gamma(\PHI_n))|\leq M$.
Then we can think of $\GG(n,M)$ as being obtained from $\Gamma(\PHI_n)$ by adding $M-|E(\Gamma(\PHI_n))|$ additional random edges (without replacement).
Hence, we have 
\begin{align}\label{eq_lem_exX2}
	\sum_{x\in V(\PHI_n)}|C(\Gamma(\PHI_n),x)|^2&\leq\sum_{v\in V(\GG(n,M))}|C(v)|^2.
\end{align}
Therefore, ~\eqref{eqAS1} and \Thm~\ref{thm_branching_est} imply that
\begin{align}\nonumber
	\ex[X(\PHI)]&\leq\sum_{x\in V(\PHI_n)}\ex\brk{|C(\Gamma(\PHI_n),x)|^2}\leq\sum_{v\in V(\GG(n,M))}\ex\brk{|C(v)|^2}\\
				  &=(1+o(1))n\sum_{k=0}^{\infty} k^2 \cdot {\eul^{-k/2} \cdot (k/2)^{k-1} \over k!}\nonumber
    \\
				  & \leq (1+o(1)){{n} \over \sqrt{2\pi}} \sum_{k=0}^{\infty} k^2 \cdot {\eul^{-k/2} \cdot (k/2)^{k-1} \over k^{k + 1/2} \cdot \eul^{-k}} &&\mbox{[by Stirling's formula]}\nonumber\\
					&= (1+o(1)){n \over \sqrt{2\pi}} \sum_{k=0}^{+\infty} \bc{\sqrt{\eul} \over 2}^k \cdot \sqrt{k}  \leq 5n.\label{eq_lem_exX3}
\end{align}
Finally, the assertion follows from \eqref{eq_lem_exX1a} and \eqref{eq_lem_exX3}.
\end{proof}

While \Lem~\ref{lem_exX} provides a bound on $\ex[X(\PHI)]$, \Prop~\ref{prop_l2} posits an upper bound that holds with high probability.
To derive this high probability bound, we combine \Lem~\ref{lem_exX} with the following variation of Azuma's inequality.

\begin{theorem}[{\cite[\Thm~2.27]{JLR}}]\label{thm_azuma}
	Let $\vZ_1, \ldots, \vZ_N$ be independent random variables, with $\vZ_k$ taking values in a set $\Lambda_k$. Assume that a function $f : \Lambda_1 \times \ldots \times \Lambda_N \ra \RR$ satisfies the following Lipschitz condition for some numbers $c_k$:

	(L) If two vectors $z, z' \in \prod_{1}^N \Lambda_i$ differ only in the $k$-th coordinate, then $|f(z) - f(z')| \leq c_k$.
	\\
	Then the random variable $\vY = f(\vZ_1, \ldots, \vZ_N)$ satisfies, for any $t \geq 0$,
	\begin{align*}
		&\Pr\brk{ \vY \geq \Erw \vY + t } \leq \exp\bc{ -{t^2 \over 2 \sum_1^N c_k^2 } },&
		&\Pr\brk{ \vY \leq \Erw \vY - t } \leq \exp\bc{ -{t^2 \over 2 \sum_1^N c_k^2 } }.
	\end{align*}
\end{theorem}

We apply \Thm~\ref{thm_azuma} to the random formula $\PHI=\PHI_{n,m}$ with its $m$ independent clauses $\va_1,\ldots,\va_m$.
Let
\begin{align*}
	\vY&=Y(\PHI)= \sum_{x\in V(\PHI)} \min\cbc{|\cV(\PHI,\cbc{x})|^2+|\cV(\PHI,\cbc{\neg x})|^2, \log^4 n}
\end{align*}
be a ``truncated" version of $X(\PHI)$.

\begin{lemma}\label{clm_azuma_bound}
	Given two 2-CNFs $\Phi,\Phi'$ with variables $x_1,\ldots,x_n$ and $m$ clauses that differ only in their $k$-th clause we have $|Y(\Phi)-Y(\Phi')| \leq 8 \log^6 n$.
\end{lemma}
\begin{proof}
	We may assume without loss that $k=m$.
	Let $a_m$ be the $m$-th clause of $\Phi$ and let $a_m'$ be the $m$-th clause of $\Phi'$.
	Further, obtain $\Phi''$ from $\Phi$ (or $\Phi'$) by removing the $m$-th clause.
	Then it suffices to prove that
	\begin{align}\label{eq_clm_azuma_bound_1}
		Y(\Phi'')\leq Y(\Phi)&\leq Y(\Phi'')+4\log^6n
	\end{align}
	because of the symmetry between $\Phi$ and $\Phi'$ and the triangle inequality.

	The first bound $Y(\Phi'')\leq Y(\Phi)$ in \eqref{eq_clm_azuma_bound_1} follows immediately because $Y(\nix)$ is monotonically increasing with respect to the addition of clauses.
	To verify the second bound, suppose that $a_m=l_1\vee l_2$.
	For a variable $x$ there are the following possible cases.
	\begin{description}
		\item[Case 1: $x \notin \cV(\Phi'',\cbc{l_1}) \cup \cV(\Phi'',\cbc{\neg l_1}) \cup \cV(\Phi'',\cbc{l_2}) \cup \cV(\Phi'',\cbc{\neg l_2})$] then we call literal $x$ \textit{neutral}.
	In this case we have
	\begin{align}\label{eq_clm_azuma_bound_2}
		|\cV(\Phi ,\cbc{x})|^2+|\cV(\Phi ,\cbc{\neg x})|^2 = |\cV(\Phi'',\cbc{x})|^2+|\cV(\Phi'',\cbc{\neg x})|^2.
	\end{align}
		\item[Case 2: there is $l_0 \in \cbc{l_1, \neg l_1, l_2, \neg l_2}$ such that $x \in \cV(\Phi'',\cbc{l_0})$ and  $|\cV(\Phi'',\cbc{l_0})| \geq \log^2 n$] then we call $x$ \textit{big} and we obtain the bound
			\begin{align}\label{eq_clm_azuma_bound_3}
                \min\cbc{|\cV(\Phi ,\cbc{x})|^2+|\cV(\Phi ,\cbc{\neg x})|^2, \log^4 n} = \min\cbc{|\cV(\Phi'',\cbc{x})|^2+|\cV(\Phi'',\cbc{\neg x})|^2, \log^4 n} = \log^4 n.
			\end{align}
		\item[Case 3:  neither Case~1 nor Case~2 occurs] then we call $x$ \textit{small} and
			\begin{align}\label{eq_clm_azuma_bound_4}
				\min\cbc{|\cV(\Phi ,\cbc{x})|^2+|\cV(\Phi ,\cbc{\neg x})|^2, \log^4 n} - \min\cbc{|\cV(\Phi'' ,\cbc{x})|^2+|\cV(\Phi'',\cbc{\neg x})|^2, \log^4 n}&\leq\log^4 n.
			\end{align}
	\end{description}
	Equations~\eqref{eq_clm_azuma_bound_2} and~\eqref{eq_clm_azuma_bound_3} show that only small variables contribute to the difference $Y(\Phi'')-Y(\Phi)$.
	Furthermore, the total number of small literals cannot exceed $4\log^2n$.
	Therefore, \eqref{eq_clm_azuma_bound_1} follows from \eqref{eq_clm_azuma_bound_4}.
\end{proof}

\begin{corollary}\label{cor_azuma_bound}
	We have, for sufficiently large $n$,
\begin{align*}
	\Pr\brk{ Y(\PHI_n) \geq \Erw[Y(\PHI_n)] + n^{0.9} } \leq \exp\bc{ - n^{0.7} }.
\end{align*}
\end{corollary}
\begin{proof}
	This is a consequence of \Lem~\ref{clm_azuma_bound} and \Thm~\ref{thm_azuma} (applied with $N=m$, $t=n^{0.9}$ and $c_k=8\log^6n$).
\end{proof}

Finally, to go back from the auxiliary random variable $Y(\PHI_n)$ to the initial $X(\PHI_n)$ we will use the following corollary.
\begin{corollary}\label{lem_subformulas_small}
	With probability $1 - o(n^{-2})$ we have
	\begin{align*}
		\max_{1 \leq i \leq n} \abs{\cV(\PHI_n, \{x_i\})} + \abs{\cV(\PHI_n, \{\neg x_i\})} \leq \log^2 n.
	\end{align*}
\end{corollary}
\begin{proof}
	This is a direct consequence of \Lem~\ref{lem_forced_tail} combined with a union bound.
\end{proof}

\begin{proof}[Proof of \Prop~\ref{prop_l2}]
	Since $X(\PHI)\geq Y(\PHI)$, \Lem~\ref{lem_exX} implies that
	\begin{align}\label{eqprop_l2_1}
		\ex\brk{Y(\PHI)}\leq\ex[X(\PHI)]\leq C(1-\alpha)^{-2}n
	\end{align}
	for some constant $C>0$.
	Furthermore, \Cor~\ref{lem_subformulas_small} shows that
	\begin{align}\label{eqprop_l2_2}
		\Pr\brk{ X(\PHI_n) \neq Y(\PHI_n)} = o(n^{-2}).
	\end{align}
	Thus, combining \eqref{eqprop_l2_1} and \Cor~\ref{cor_azuma_bound}, we obtain, for sufficiently large $n$,
	\begin{align*}
		\pr\brk{X(\PHI)>C(1-\alpha)^{-2}n+n^{0.9}}&\leq \exp(-n^{0.7})+o(n^{-2})=o(1),
	\end{align*}
	as claimed.
\end{proof}

\section{Proof of \Prop~\ref{prop_flips_per_var}}\label{sec_prop_flips_per_var}

\noindent
It suffices to show that if $t<\vT(\Phi)$ and $\fE(\Phi,l,\SIGMA^{(t)})$ occurs, then $\fE(\Phi,l,\SIGMA^{(t+1)})$ occurs.
Clearly, the only way $\fE(\Phi,l,\SIGMA^{(t+1)})$ could possibly fail to occur is because \WS\ flipped a literal $l'\in\cL(\Phi,\cbc l)$ in step $t+1$.
In this case $\Phi$ contains a clause $a=\neg l'\vee l''$ with
\begin{align}\label{eq_sec_prop_flips_per_var1}
	\SIGMA^{(t)}(l'')=-1,
\end{align}
because $\fE(\Phi,l,\SIGMA^{(t)})$ occurs and hence $\SIGMA^{(t)}(l')=1$.
But since $l'\in\cL(\Phi,\cbc l)$, the construction of $\cL(\Phi,\cbc l)$ by the way of the \UCP\ (Algorithm~\ref{fig_ucp}) ensures that $l''\in\cL(\Phi,\{l\})$.
Consequently, if $\fE(\Phi,l,\SIGMA^{(t)})$ occurs, then $\SIGMA^{(t)}(l'')=1$.
This contradiction to~\eqref{eq_sec_prop_flips_per_var1} refutes the assumption that $\fE(\Phi,l,\SIGMA^{(t+1)})$ fails to occur.

\section{Proof of \Prop~\ref{prop_upper_bound_small_subformula}}\label{sec_prop_upper_bound_small_subformula}

\noindent
We are going to adapt the well-known random walk argument from~\cite{WalkSAT1}, an excellent exposition of which can be found in \cite[Chapter~ 6]{MR}, to the present task.
Fix a variable $x \in V$.
Recalling that $\sigma^*$ is a satisfying assignment of $\Phi$, we consider
	\begin{align}\label{eq_prop_upper_bound_small_subformula_0}
		\vec\Delta^*(\Phi,x,t)&={\vecone\{\SIGMA^{(t)}\not\models\Phi\}}\cdot \sum_{l\in\cL(\Phi,\{\sigma^*(x)\cdot x\})}\vecone\cbc{\SIGMA^{(t)}(l)\neq\sigma^*(l)}.
	\end{align}
In words, $\vec\Delta^*(\Phi,x,t)$ equals the number of literals in the implication subformula of $\sigma^*(x)\cdot x$ whose truth value at time $t$ disagrees with the truth value under $\sigma^*$, and if \WS\ reached $\sigma^*$ or any other satisfying assignment, we set $\vec\Delta^*(\Phi,x,t) = 0$.
Let $\vT^*(\Phi,x)$ be the least $t$ such that $\vec\Delta^*(\Phi,x,t)=0$ (and $\infty$ if $\vec\Delta^*(\Phi,x,t)>0$ for all $t\geq0$).
\Prop~\ref{prop_flips_per_var} implies that
\begin{align}\label{eq_prop_upper_bound_small_subformula_1}
	\vec\Delta^*(\Phi,x,t)&=0&&\mbox{if }\vT^*(\Phi,x)\leq t<\infty.
\end{align}
Additionally, we recall from the pseudocode displayed as Algorithm~\ref{fig_walksat} that the clause that \WS\ picks a time $t$ is called $\va^{(t)}=\vl^{(t)}_1\vee\vl^{(t)}_2$ and that the literal that \WS\ flips is called $\vl^{(t)}_{\vh^{(t)}}$.
With this in mind we have
\begin{align}\label{eq_prop_upper_bound_small_subformula_100}
	\vN(\Phi,\sigma^*(x)\cdot x)&=\sum_{t=1}^{\vT^*(\Phi,x)}\vecone\cbc{|\vl^{(t)}_{\vh^{(t)}}|\in\cV(\Phi,\{\sigma^*(x)\cdot x\})},
\end{align}
because~\eqref{eq_prop_upper_bound_small_subformula_1} shows that the sum on the r.h.s.\ counts the number of times that \WS\ flips a variable from $\cV(\Phi,\{\sigma^*(x)\cdot x\})$.

In order to estimate this sum we are going to argue that
\begin{align}\label{eq_prop_upper_bound_small_subformula_2}
	\ex\brk{\vec\Delta^*(\Phi,x,t+1)\mid\SIGMA^{(t)}}&\leq\vec\Delta^*(\Phi,x,t)&&\mbox{for all }0 \leq t< \vT^*(\Phi, x),
\end{align}
i.e., that the sequence $(\vec\Delta^*(\Phi,x,t))_t$ forms a supermartingale.
To see this, we consider several cases.
\begin{description}
	\item[Case 1: $|\vl_{1}^{(t)}|,|\vl_{2}^{(t)}|\not\in\cV(\Phi,\{\sigma^*(x)\cdot x\})$]
		then $|\vl_{\vh^{(t)}}^{(t)}|\not\in\cV(\Phi,\{\sigma^*(x)\cdot x\})$ and thus
		the definition~\eqref{eq_prop_upper_bound_small_subformula_0} of the random variable ensures that $\vec\Delta^*(\Phi,x,t+1)=\vec\Delta^*(\Phi,x,t)$.
	\item[Case 2: $|\vl_{1}^{(t)}|\in\cV(\Phi,\{\sigma^*(x)\cdot x\})$ but $|\vl_{2}^{(t)}|\not\in\cV(\Phi,\{\sigma^*(x)\cdot x\})$]
		in this case we know that $\SIGMA^{(t)}(\vl^{(t)}_{1})=-1$, because otherwise $\va^{(t)}$ would be satisfied.
		We also see that $\sigma^*(\vl^{(t)}_{1})=1$, because otherwise \UCP\ (Algorithm~\ref{fig_ucp}) would have ensured that $\vl_2^{(t)}\in\cL(\Phi,\{\sigma^*(x)\cdot x)$.
		Hence, we obtain
		\begin{align}\label{eq_prop_upper_bound_small_subformula_3}
			\vec\Delta^*(\Phi,x,t+1)=\vec\Delta^*(\Phi,x,t)-\vecone\{\vh^{(t)}=1\}.
		\end{align}
	\item[Case 3: $|\vl_{1}^{(t)}|\not\in\cV(\Phi,\{\sigma^*(x)\cdot x\})$ but $|\vl_{2}^{(t)}|\in\cV(\Phi,\{\sigma^*(x)\cdot x\})$]
		as in Case~2 we obtain
		\begin{align}\label{eq_prop_upper_bound_small_subformula_4}
			\vec\Delta^*(\Phi,x,t+1)=\vec\Delta^*(\Phi,x,t)-\vecone\{\vh^{(t)}=2\}.
		\end{align}
	\item[Case 4: $|\vl_{1}^{(t)}|,|\vl_{2}^{(t)}|\in\cV(\Phi,\{\sigma^*(x)\cdot x\})$]
		since $\sigma^*(\vl_{1}^{(t)})=1$ or $\sigma^*(\vl_{2}^{(t)})=1$ while $\SIGMA^{(t)}(\vl_{1}^{(t)})=\SIGMA^{(t)}(\vl_{2}^{(t)})=-1$ (because $\va^{(t)}$ is unsatisfied), the random choice of $\vh^{(t)}\in\{1,2\}$ ensures that
\begin{align}\label{eq_prop_upper_bound_small_subformula_5}
	\pr\brk{\vec\Delta^*(\Phi,x,t+1)=\vec\Delta^*(\Phi,x,t)-1\mid\SIGMA^{(t)}}&\geq\frac12,&&\mbox{while}\\
	\pr\brk{\vec\Delta^*(\Phi,x,t+1)=\vec\Delta^*(\Phi,x,t)+1\mid\SIGMA^{(t)}}&\leq\frac12.\label{eq_prop_upper_bound_small_subformula_6}
\end{align}
\end{description}
Combining \eqref{eq_prop_upper_bound_small_subformula_3}--\eqref{eq_prop_upper_bound_small_subformula_6} and keeping in mind that $\vec\Delta^*(\Phi,x,t+1)\in\{\vec\Delta^*(\Phi,x,t)-1,\vec\Delta^*(\Phi,x,t),\vec\Delta^*(\Phi,x,t)+1\}$, we obtain \eqref{eq_prop_upper_bound_small_subformula_2}.

Additionally, \eqref{eq_prop_upper_bound_small_subformula_3}--\eqref{eq_prop_upper_bound_small_subformula_6} imply that for $t<\vT^*(\Phi,x)$,
\begin{align}\label{eq_prop_upper_bound_small_subformula_7}
	\pr\brk{\vec\Delta^*(\Phi,x,t+1)\neq\vec\Delta^*(\Phi,x,t)\mid\SIGMA^{(t)},\va^{(t)}}&\geq\frac12\sum_{h=1}^2\vecone\{|\vl_{\vh^{(t)}}^{(t)}|\in\cV(\Phi,\{\sigma^*(x)\cdot x\}), \vh^{(t)} = h\},&&\mbox{while}\\
	\pr\brk{\vec\Delta^*(\Phi,x,t+1)\neq\vec\Delta^*(\Phi,x,t)\mid\SIGMA^{(t)},\va^{(t)}}&
	\leq\max_{h=1,2}\vecone\{|\vl_{\vh^{(t)}}^{(t)}|\in\cV(\Phi,\{\sigma^*(x)\cdot x\}), \vh^{(t)} = h\}&&\mbox{and}\label{eq_prop_upper_bound_small_subformula_8}\\
		\vec\Delta^*(\Phi,x,t)&\leq|\cL(\Phi,\{\sigma^*(x)\cdot x\})|. \label{eq_prop_upper_bound_small_subformula_9}
\end{align}
Combining~\eqref{eq_prop_upper_bound_small_subformula_100}--\eqref{eq_prop_upper_bound_small_subformula_2} with \eqref{eq_prop_upper_bound_small_subformula_7}--\eqref{eq_prop_upper_bound_small_subformula_9}, we see that $\ex[\vN(\Phi,x)]$ is dominated by the expected time that a simple random walk started at the origin reaches the value $-|\cL(\Phi,\{\sigma^*(x)\cdot x\})|$ for the first time.
This expectation is well known to be bounded by $C\cdot|\cL(\Phi,\{\sigma^*(x)\cdot x\})|^2$ for a constant $C>0$.

\end{document}